\documentclass[a4paper,11pt]{article}
\usepackage{amsfonts}
\usepackage{enumerate}
\usepackage{graphicx}
\usepackage{hyperref}
\usepackage{amsmath,amssymb,amsthm}
\usepackage{cleveref}
\usepackage{url}
\usepackage{amscd}
\newcommand{\be}{\begin{equation}}
\newcommand{\ee}{\end{equation}}

\def\x{{\bf x}}

\def\fa{{\bf f}}

\def\I{\mathcal{I}}

\newtheorem{theorem}{\bf Theorem}[section]
\newtheorem{lemma}{\bf Lemma}[section]
\newtheorem{defi}{\bf Definition}[section]
\newtheorem*{claim}{\bf Claim}
\newtheorem{remark}{\bf Remark}[section]
\newtheorem{example}{\bf Example}[section]
\numberwithin{equation}{section}

\newtheorem{thmx}{Theorem}

\newtheorem*{main}{Main Theorem}

\begin{document}
%\begin{CJK*}{GBK}{song}

\title{\bf Cartan's Conjecture for Moving Hypersurfaces}
\author{Qiming Yan\footnotemark[1] \and Guangsheng Yu\footnotemark[2]}
\footnotetext[1]{Department of Mathematics,	Tongji University, Shanghai 200092, P. R. China;}
\footnotetext[2]{Department of Mathematics, University of Shanghai for Science and Technology, Shanghai 200093, P. R. China. Corresponding author.}
\footnotetext{Email Address : %yan_qiming@hotmail.com;
		yan$\underline{\mbox{ }}$qiming@hotmail.com;
		ygsh@usst.edu.cn.}
\footnotetext{Q. Yan was partially supported by NSFC11571256; G. Yu was partially supported by NSFC11671090.}
\date{}
\maketitle

{\bf Abstract:} Let $f$ be a holomorphic curve in $\mathbb{P}^n({\mathbb{C}})$ and let $\mathcal{D}=\{D_1,\ldots,D_q\}$ be a family of moving hypersurfaces
defined by a set of homogeneous polynomials $\mathcal{Q}=\{Q_1,\ldots,Q_q\}$. For $j=1,\ldots,q$, denote by $Q_j=\sum\limits_{i_0+\cdots+i_n=d_j}a_{j,I}(z)x_0^{i_0}\cdots x_n^{i_n}$, where $I=(i_0,\ldots,i_n)\in\mathbb{Z}_{\ge 0}^{n+1}$ and  $a_{j,I}(z)$  are entire functions on ${\mathbb{C}}$ without common zeros.
Let $\mathcal{K}_{\mathcal{Q}}$ be the smallest subfield of meromorphic function field $\mathcal{M}$ which contains ${\mathbb{C}}$ and all $\frac{a_{j,I'}(z)}{a_{j,I''}(z)}$ with $a_{j,I''}(z)\not\equiv 0$, $1\le j\le q$. In previous known second main theorems for $f$ and $\mathcal{D}$, $f$ is usually assumed to be algebraically nondegenerate over $\mathcal{K}_{\mathcal{Q}}$. In this paper, we prove a second main theorem in which $f$ is only assumed to be nonconstant. This result can be regarded as a generalization of Cartan's conjecture for moving hypersurfaces.



\section{Introduction}

Nevanlinna theory begins with two main theorems (known as
Nevanlinna's first and second main theorems) for meromorphic functions on $\mathbb{C}$, which were established by
R. Nevanlinna in the 1920's (see \cite{1}). The following second main theorem is the heart of Nevanlinna theory. (Here, we use some notations which will be introduced below.)

%\noindent{\bf Theorem A. (Nevanlinna's second main theorem)}\label{thma}
%\emph{Let $f$ be a nonconstant meromorphic function on $\mathbb{C}$. Then,
%for any distinct points $a_1, \dots, a_q\in  \mathbb{C}\cup\{\infty\}$
%and any $\varepsilon>0$,
%$$\|\sum_{j=1}^q m_f(r, a_j) \leq (2+\varepsilon)T_f(r), $$
%where ``$\|$'' means that the inequality holds for all $r$ outside a set with finite
%Lebesgue measure.}

\begin{thmx}[Nevanlinna's second main theorem]\label{thma}
	Let $f$ be a nonconstant meromorphic function on $\mathbb{C}$. Then,
	for any distinct points $a_1, \dots, a_q\in  \mathbb{C}\cup\{\infty\}$
	and any $\varepsilon>0$,
	$$\|\sum_{j=1}^q m_f(r, a_j) \leq (2+\varepsilon)T_f(r), $$
	where ``$\|$'' means that the inequality holds for all $r$ outside a set with finite
	Lebesgue measure.
\end{thmx}

At the same time, Nevanlinna conjectured that Theorem \ref{thma} should
remain valid if the fixed points $a_j$  were replaced by slowly
growth  meromorphic functions $a_j(z)$, which is the so-called moving targets problem. Osgood \cite{2} proved this conjecture and Steinmetz \cite{3} gave
another simple and elegant proof.

\begin{thmx}[Nevanlinna's conjecture]\label{thmb}
Let $f$ be a nonconstant meromorphic function on $\mathbb{C}$, and let $a_j(z)$, $1\le j\le q$, be slowly growth meromorphic functions with respect to $f$ (i.e., $T_{a_j}(r)=o(T_f(r))$).
Then, for any $\varepsilon>0$,
$$\|\sum_{j=1}^q m_f(r, a_j) \leq (2+\varepsilon)T_f(r). $$
\end{thmx}

During the last few decades, there were several generalizations of Theorem \ref{thmb} for higher dimensional case. To state these results, we first recall some notations and definitions in Nevanlinna theory.

%Let $f=[f_0:\cdots:f_n]:{\mathbb{C}}\rightarrow
%{\mathbb{P}}^n({\mathbb{C}})$ be a holomorphic map, where
%$f_0,\ldots,f_n$ are entire functions on ${\mathbb{C}}$ without
%common zeros. (i.e., $f$ is a holomorphic curve in ${\mathbb{P}}^n({\mathbb{C}})$.) Define ${\bf{f}}=(f_0,\ldots,f_n)$. ${\bf{f}}$ is called a
%reduced representation of $f$. The characteristic function of $f$ is
%defined by
%$$
%T_f(r)=\int_0^{2\pi}\log\|{\bf{f}}(re^{i\theta})\|\frac{d\theta}{2\pi},
%$$
%where $\|{\bf{f}}(z)\|=\max\{|f_0(z)|,\ldots,|f_n(z)|\}$.
Let $f:{\mathbb{C}}\rightarrow
{\mathbb{P}}^n({\mathbb{C}})$ be a holomorphic map, the characteristic function of $f$ is
defined by
$$
T_f(r)=\int_0^{2\pi}\log\|{\bf{f}}(re^{i\theta})\|\frac{d\theta}{2\pi},
$$
where ${\bf{f}}=(f_0,\ldots,f_n)$ is a reduced representation of $f$ with $f_0,\ldots,f_n$ having no common zeros and $\|{\bf{f}}(z)\|=\max\{|f_0(z)|,\ldots,|f_n(z)|\}$.
In particular, for a meromorphic function $f$ on ${\mathbb{C}}$, we can choose two holomorphic functions
$f_0,f_1$ on ${\mathbb{C}}$ without common zeros such that $f$ can be regarded as a holomorphic map from $\mathbb{C}$ to $\mathbb{P}^1(\mathbb{C})$ with a reduced representation $(f_0,f_1)$, and then we can define the characteristic function of $f$.

Note that a divisor on ${\mathbb{P}}^n({\mathbb{C}})$ is a hypersurface defined by some homogeneous polynomial. Now, we introduce the so-called moving targets on ${\mathbb{P}}^n({\mathbb{C}})$. For a positive integer $d$, we set
$$
\I_d:=\{I=(i_0,\ldots,i_n)\in\mathbb{Z}_{\ge 0}^{n+1}|i_0+\cdots+i_n=d\}
$$
and
$$ n_d=\sharp\I_d=\left(\begin{matrix}d+n\\n\end{matrix}\right).$$
A \textbf{moving hypersurface} $D$ in  $\mathbb{P}^n({\mathbb{C}})$ of degree $d$ is defined by a homogenous polynomial $Q=\sum_{I\in\I_{d}}a_{I}\x^I$,
where $a_I$, $I\in\I_{d}$, are holomorphic functions on ${\mathbb{C}}$ without common zeros, and $\x^I=x_0^{i_0}\cdots x_n^{i_n}$. Then, for each $z\in {\mathbb{C}}$, $D(z)$ is a (fixed) hypersurface in $\mathbb{P}^n({\mathbb{C}})$ defined by the zero set of $Q(z)=\sum_{I\in\I_{d}}a_{I}(z)\x^I$. If $d=1$, $D$ is called a moving hyperplane. Since a moving hypersurface $D$
can be regarded as a holomorphic map $\mathfrak{D}: {\mathbb{C}}\rightarrow\mathbb{P}^{n_d-1}({\mathbb{C}})$ with a reduced representation $(\ldots,a_I(z),\ldots)_{I\in\I_d}$, we call $D$ a {\textbf{slowly moving hypersurface}} with respect to $f$ if $T_{\mathfrak{D}}(r)=o(T_f(r))$.

The proximity function of $f$ with respect to the moving hypersurface $D$ is defined by
$$
m_f(r,D)=\int_0^{2\pi}\log\lambda_D(\fa(re^{i\theta}))\frac{d\theta}{2\pi},
$$
where $\lambda_D(\fa)=\log\frac{\|\fa\|^d\|Q\|}{|Q(\fa)|}$ is the Weil function and $\|Q\|=\max\limits_{I\in\I_d}\{|a_I|\}$. (Sometimes, we also write $\lambda_{D}(\fa)$ as $\lambda_{Q}(\fa)$.)

Let ${\mathcal D}=\{D_1,\ldots, D_q\}$ be a family of moving hypersurfaces in $\mathbb{P}^n({\mathbb{C}})$. We say that $D_1,\ldots, D_q$ are in \textbf{$m$-subgeneral position} ($m\ge n$) if there exists $z\in {\mathbb{C}}$ such that $D_1(z),\ldots, D_q(z)$ are in $m$-subgeneral position (as fixed hypersurfaces), i.e., any $m+1$ of $D_1(z),\ldots, D_q(z)$ do not meet at one point. Actually, if this condition is satisfied for one $z\in {\mathbb{C}}$, it is also satisfied for all $z$ except for a discrete set. We say that $D_1,\ldots, D_q$ are \textbf{in general position} if they are in $n$-subgeneral position.

Let ${\mathcal Q}=\{Q_1,\ldots, Q_q\}$ be the family of homogeneous polynomials with $\deg Q_j=d_j$ defining ${\mathcal D}=\{D_1,\ldots, D_q\}$. Assume that
$$
Q_j=\sum_{I\in\I_{d_j}}a_{j,I}(z)\x^I,\ 1\le j\le q.
$$
We denote by $\mathcal{K}_{\mathcal{Q}}$ the smallest subfield of meromorphic function field $\mathcal{M}$ which contains ${\mathbb{C}}$ and all $\frac{a_{j,I_s}}{a_{j,I_t}}$ with $a_{j,I_t}\not\equiv 0$, $j=1,\ldots,q$, $I_s,I_t\in\I_{d_j}$. We say that $f$ is \textbf{linearly nondegenerate over $\mathcal{K}_{\mathcal{Q}}$} if there is no nonzero linear form $L\in\mathcal{K}_{\mathcal{Q}}[x_0,\ldots,x_n]$ such that $L(f_0,\ldots,f_n)\equiv 0$, and $f$ is \textbf{algebraically nondegenerate over $\mathcal{K}_{\mathcal{Q}}$} if there is no nonzero homogeneous polynomial $Q\in\mathcal{K}_{\mathcal{Q}}[x_0,\ldots,x_n]$ such that $Q(f_0,\ldots,f_n)\equiv 0$.

As a generalization of Theorem \ref{thma}, H. Cartan \cite{4} proved a second main theorem for linearly nondegenerate (over ${\mathbb{C}}$) holomorphic curves intersecting (fixed) hyperplanes in $\mathbb{P}^n({\mathbb{C}})$, and posed a conjecture for nonconstant holomorphic curves which was solved by Nochka (see \cite{5}). For the moving targets case, Ru and Stoll \cite{6} generalized Theorem \ref{thmb} to holomorphic curves intersecting moving hyperplanes in $\mathbb{P}^n({\mathbb{C}})$ as follows.

\begin{thmx}\label{thmc}
	Let $f:{\mathbb{C}}\rightarrow\mathbb{P}^n({\mathbb{C}})$ be a nonconstant holomorphic curve, and let ${\mathcal D}=\{D_1,\ldots, D_q\}$ be a family of slowly moving hyperplanes with respect to $f$ with the set of defining linear forms ${\mathcal Q}=\{Q_1,\ldots, Q_q\}$.

(i) (Cartan-Nochka type theorem) Assume that $f$ is linearly nondegenerate over $\mathcal{K}_{\mathcal{Q}}$ and $D_1,\ldots,D_q$ are located in $m$-subgeneral position. Then, for any $\varepsilon>0$,
$$\|\sum_{j=1}^q m_f(r, D_j) \leq (2n-m+1+\varepsilon)T_f(r). $$

(ii) (Cartan's conjecture for moving hyperplanes) Assume that $D_1,\ldots,D_q$ are located in general position such that $Q_j(\fa)\not\equiv 0$, $j=1,\ldots,q$. Then, for any $\varepsilon>0$,
$$\|\sum_{j=1}^q m_f(r, D_j) \leq (2n+\varepsilon)T_f(r). $$
\end{thmx}

In 1992, Eremenko and Sodin \cite{7} gave a generalization of (ii) of Theorem \ref{thmc} for moving hypersurfaces under a stronger assumption $T_{\mathfrak {D}_j}(r)=o(\frac{T_f(r)}{\log^{\tau}T_f(r)})$ with $\tau>1$ (see Theorem 3 in \cite{7}). Recently, Si \cite{8} obtained the following second main theorem for slowly moving hypersurfaces in subgeneral position.

\begin{thmx}\label{thmd}
	Let $f:{\mathbb{C}}\rightarrow\mathbb{P}^n({\mathbb{C}})$ be a nonconstant holomorphic curve. Let ${\mathcal D}=\{D_1,\ldots, D_q\}$ be a family of slowly (with respect to $f$) moving hypersurfaces in $m$-subgeneral position, and let ${\mathcal Q}=\{Q_1,\ldots, Q_q\}$ be the set of defining homogeneous polynomials of ${\mathcal D}$ with $\deg Q_j=d_j$ ($j=1,\ldots,q$). Assume that $f$ is algebraically nondegenerate over $\mathcal{K}_{\mathcal{Q}}$. Then, for any $\varepsilon>0$,
$$\|\sum_{j=1}^q \frac{1}{d_j}m_f(r, D_j) \leq (m-n+1)(n+1+\varepsilon)T_f(r). $$
\end{thmx}

\begin{remark}\label{remark1.1}
	When $m=n$, Theorem \ref{thmd} is obtained by Dethloff and Tran \cite{9}. Actually, under the assumption that those moving hypersurfaces are in general position, Dethloff and Tran \cite{9-1} obtained a more general second main theorem as follows.
\end{remark}

\begin{thmx}\label{thme}
	Let $f:{\mathbb{C}}\rightarrow V\subset\mathbb{P}^M({\mathbb{C}})$ be a nonconstant holomorphic curve, where $V$ is an irreducible algebraic subvariety of dimension $n$. Let ${\mathcal D}=\{D_1,\ldots, D_q\}$ be a family of slowly (with respect to $f$) moving hypersurfaces in $\mathbb{P}^M({\mathbb{C}})$ which are in general position on $V$ (i.e., there exists $z\in \mathbb{C}$ such that $D_1(z),\ldots, D_q(z)$ are in general position on $V$). Let ${\mathcal Q}=\{Q_1,\ldots, Q_q\}$ be the set of defining homogeneous polynomials of ${\mathcal D}$ with $\deg Q_j=d_j$. Assume that $f$ is algebraically nondegenerate over $\mathcal{K}_{\mathcal{Q}}$. Then, for any $\varepsilon>0$,
$$\|\sum_{j=1}^q \frac{1}{d_j}m_f(r, D_j) \leq (n+1+\varepsilon)T_f(r). $$
\end{thmx}

Note that the condition ``$f$ is algebraically nondegenerate over $\mathcal{K}_{\mathcal{Q}}$" in Theorems \ref{thmd} and \ref{thme} is difficult to check. The purpose of this paper is to give a second main theorem as a generalization of Cartan's conjecture for moving hypersurfaces, in which we only assume that $f$ is nonconstant and the moving hypersurfaces are in subgeneral position.

\begin{main}\label{mainthm}
	Let $f$ be a nonconstant holomorphic curve in $\mathbb{P}^n({\mathbb{C}})$ ($n>1$). Let ${\mathcal D}=\{D_1,\ldots, D_q\}$ be a family of slowly (with respect to $f$) moving hypersurfaces in $m$-subgeneral position, and let ${\mathcal Q}=\{Q_1,\ldots, Q_q\}$ be the set of defining homogeneous polynomials of ${\mathcal D}$ with $\deg Q_j=d_j$ and $Q_j(\fa)\not\equiv 0$ for $j=1,\ldots,q$. Then, for any $\varepsilon>0$,
\begin{eqnarray}
\|\sum_{j=1}^q \frac{1}{d_j}m_f(r, D_j) \leq ((m-\min\{n,m/2\}+1)(\min\{n,m/2\}+1)+\varepsilon)T_f(r).
\label{eq1.1}
\end{eqnarray}
\end{main}
%（不等式右端须改。）

\section{Preliminaries on algebraic geometry over general fields}\label{algg}

In classical algebraic geometry setting, the field is always $\mathbb{C}$. However, in order to prove our main theorem, we need the algebraic geometry over an arbitrary base field and the coordinates of points are taken over an arbitrary extension field of the base field. In this section, we list some basic facts of the theory of algebraic geometry over general fields. For more details, please refer to \cite{We46}, \cite[Chp 16]{10-3}  and \cite[Chp VII]{ZS60}.

Let $k$ be a field and let $K$ be an algebraically closed extension of $k$. We call $k$ a base field and $K$ the coordinate domain. In this paper, we take $K$ to be the universal field $\Omega$ over the base field $k$, which means that $\Omega$ is algebraically closed and has infinite transcendence degree over $k$. A useful fact of $\Omega$ is that any field extension obtained by adjoining finitely many field elements to $k$ can be isomorphically imbedded in $\Omega$ which leaves the elements of $k$ fixed.

The points of the $n$-dimensional projective space $\mathbb{P}^n(\Omega)$ (over $\Omega$) are represented by ordered $(n+1)$-tuples $(y_0,\ldots,y_n)$ of elements of $\Omega$ with the $(n+1)$-tuple $(0,\ldots,0)$ being excluded and two $(n+1)$-tuples $(y_0,\ldots,y_n)$ and $(y'_0,\ldots,y'_n)$ representing the same point $p$ if and only if there exists an element $t\neq 0$ in $\Omega$ such that $y'_i=ty_i,i=0,\ldots,n$. The $(n+1)$-tuple $(y_0,\ldots,y_n)$ is called a set of homogeneous coordinates of the corresponding point.

Let $F(x_0,\ldots,x_n)$ be a homogeneous polynomial over $k$ and $p$ be a point of $\mathbb{P}^n(\Omega)$. If there is a set of homogeneous coordinates $(y_0,\ldots,y_n)$ of $p$ satisfies $F(y_0,\ldots,y_n)$, then every set of homogeneous coordinates $(y'_0,\ldots,y'_n)$ of $p$ will satisfy $F(y'_0,\ldots,y'_n)=0$. We say that the point $p$ is a zero of $F$ and that $F$ vanishes at $p$.

If $I$ is a homogeneous ideal in $k[x_0,\ldots,x_n]$, any common zero of $F$ belonging to $I$ is called a zero of the ideal $I$, and the set of zeros of $I$ is called a variety of $I$ and is denoted by $V(I)$. A projective algebraic variety $V$ in $\mathbb{P}^n(\Omega)$, defined over $k$, is any subset of $\mathbb{P}^n(\Omega)$ which is the variety of some homogeneous ideal in $k[x_0,\ldots,x_n]$.

If $E$ is any subset of $\mathbb{P}^n(\Omega)$, then the set of homogeneous polynomials in $k[x_0,\ldots,x_n]$ which vanish at every point of $E$ is obviously the set of homogeneous polynomials belonging to a well-defined homogeneous ideal. This homogeneous ideal is called the ideal of the set $E$ and will be denoted by $I(E)$.

We have a natural topology on the projective space $\mathbb{P}^n(\Omega)$ in which the projective algebraic varieties are the closed sets.

The closure of any subset $E$ of $\mathbb{P}^n(\Omega)$, i.e., the smallest variety containing $E$, is given by $V(I(E))$.

\begin{example}
	Let $\mathcal{K}_\mathcal{Q}$ be the base field and $\Omega$ be the universal field over $\mathcal{K}_\mathcal{Q}$. Let $f:\mathbb{C}\rightarrow\mathbb{P}^n(\mathbb{C})$ be a holomorphic map. A reduced representation $(f_0,\ldots,f_n)$ of $f$ can be regarded as a set of homogeneous coordinates of some point $p$ in $\mathbb{P}^n(\Omega)$. Now we can give an algebraic geometry explanation of algebraically nondegenerate condition over $\mathcal{K}_\mathcal{Q}$, that is, $f$ is algebraically nondegenerate over $\mathcal{K}_\mathcal{Q}$ means that the closure of $p$ is $\mathbb{P}^n(\Omega)$.
	
\end{example}

Next, we introduce the definition of dimension of a projective algebraic variety in $\mathbb{P}^n(\Omega)$. Let $V$ be a nonempty irreducible variety in $\mathbb{P}^n(\Omega)$, i.e., the homogeneous ideal of $V$ is a prime ideal in $k[x_0,\ldots,x_n]$. Denote by $I$ this homogeneous ideal. We say that $p$ is a generic point of $I$ if $F\in I$ implies $F$ vanishes at $p$ and conversely. If $p$ is a generic point and $(y_0,\ldots,y_n)$ is a set of homogeneous coordinates of $p$, denote by $k(p)$ the field generated over $k$ by all the ratios $\frac{y_i}{y_j}$ such that $y_j\neq 0$, which is independent of the choice of the set of homogeneous coordinates of $p$. The dimension of an irreducible variety $V\subset\mathbb{P}^n(\Omega)$ is an integer between $0$ and $n$, and is defined by the transcendence degree of the generic point with respect to its homogeneous prime ideal, i.e., $\dim V:=$the transcendence degree of $k(p)/k$. Since every variety can be uniquely decomposed as a finite sum of irreducible varieties, the dimension of an arbitrary variety is defined to be the highest of the dimensions of its irreducible components.

\begin{remark}\label{remark2.1}
	We remark that this definition of dimension agrees with other more frequently used definitions, e.g., the supremum of all integers $m$ such that there exists a chain $Z_0\subset Z_1\subset\ldots\subset Z_m$ of distinct irreducible closed subsets of $V$.%the maximal length of the chains of distinct nonempty closed irreducible subvarieties of $V$.
\end{remark}

Let $I\subset k[x_0,\ldots,x_n]$ be a homogeneous ideal. The Hilbert function of $I$ is
\begin{eqnarray*}
h_I&:&\mathbb{Z}\rightarrow\mathbb{Z} \\
&&N\mapsto\dim_k\left(k[x_0,\ldots,x_n]_N/I_N\right),
\end{eqnarray*}where $k[x_0,\ldots,x_n]_N$ is the vector space of homogeneous polynomials of degree $N$ and $I_N:=I\cap k[x_0,\ldots,x_n]_N$. An important fact of Hilbert function is that there exist a polynomial $P_I\in\mathbb{Q}[t]$ of degree $d$ and a positive integer $N_0$ such that
\begin{equation*}
h_I(N)=P_I(N)\quad\text{for}\ N\geq N_0.
\end{equation*}
This polynomial $P_I:=a_dt^d+\cdots+a_1t+a_0$ is called the Hilbert polynomial of $I$. Let $V\subset\mathbb{P}^n(\Omega)$ be the projective variety of $I$, then $d=\dim V$. (For the proof, see \cite[Theorems 41 and 42 of chapter VII]{ZS60}.) Define the degree of $V$ by $\deg V:= a_d\cdot d!\in \mathbb{N}$.

Here we recall the relations between affine and projective varieties. Denote by $\mathbb{A}^n(\Omega)$ and $\mathbb{P}^n(\Omega)$ the affine space and projective space. There is a natural map $\varphi:\mathbb{A}^n(\Omega)\rightarrow\mathbb{P}^n(\Omega)$ which maps a point $p\in\mathbb{A}^n(\Omega)$ with coordinates $(y_1,\ldots,y_n)$ to a point $\varphi(p)\in\mathbb{P}^n(\Omega)$ with homogeneous coordinates $(1,y_1,\ldots,y_n)$. This map $\varphi$ identifies the affine space $\mathbb{A}^n(\Omega)$ with the complement of the hyperplane at infinity $\{x_0=0\}$ in the projective space $\mathbb{P}^n(\Omega)$. The connection between projective and affine varieties is as follows: If $V$ is a variety (not contained in the hyperplane at infinity) in $\mathbb{P}^n(\Omega)$, then the intersection of $V$ with $\mathbb{A}^n(\Omega)$ is an affine variety with the same dimension.

\section{Proof of the Main Theorem}

Replacing $Q_j$ by $Q_j^{d/d_j}$ if necessary, where $d$ is the least common multiple of $d_j$'s, we can assume that $Q_1,\ldots,Q_q$ have the same degree $d$. Set
$$
Q_j(z)=\sum_{I\in\I_{d}}a_{j,I}(z)\x^I,\ j=1,\ldots,q.
$$
For each $j$, there exists  $a_{j,I^j}(z)$, one of the coefficients in $Q_j(z)$, such that $a_{j,I^j}(z)\not\equiv0$. We fix this $a_{j,I^j}$, then set $\widetilde{a}_{j,I}(z)=\frac{a_{j,I}(z)}{a_{j,I^j}(z)}$ and
$$
\widetilde{Q}_j(z)=\sum_{I\in\I_d}\widetilde{a}_{j,I}(z)\x^I,
$$
which is a homogeneous polynomial in $\mathcal{K}_{\mathcal{Q}}[x_0,\ldots,x_n]$. By definition, we have
$$
\lambda_{D_j}(\fa)=\log\frac{\|\fa\|^d\|Q_j\|}{|Q_j(\fa)|}=\log\frac{\|\fa\|^d\|\widetilde{Q}_j\|}{|\widetilde{Q}_j(\fa)|}\ \mbox{for}\ j=1,\ldots,q.
$$

Denote by $\mathcal{C}_\mathcal{Q}$ the set of all nonnegative functions $h:{\mathbb{C}}\rightarrow [0,+\infty]\subset \overline{\mathbb{R}}$, which are of the form
$$
\frac{|g_{1}|+\cdots+|g_s|}{|g_{s+1}|+\cdots+|g_t|},
$$
where $s,t\in {\mathbb{N}}$, $g_1,\ldots,g_t\in \mathcal{K}_{\mathcal{Q}}\setminus\{0\}$. It is easy to see that the sums, products and quotients of functions in $\mathcal{C}_\mathcal{Q}$ are again in $\mathcal{C}_\mathcal{Q}$. Obviously, for any $h\in \mathcal{C}_\mathcal{Q}$, we have
$$
\int_0^{2\pi}\log h(re^{i\theta})\frac{d\theta}{2\pi}=o(T_f(r)).
$$

Since ${\mathcal D}$ is a family of moving hypersurfaces which are in $m$-subgeneral position, we have the following lemma.

%\textbf{Lemma 3.1.}\label{lem3.1} \textit{For any $D_{j_1},\ldots,D_{j_{m+1}}\in {\mathcal D}$, there exist functions $h_1,h_2\in \mathcal{C}_\mathcal{Q}$ such that
%$$
%h_2\|\fa\|^d\le \max_{l\in\{1,\ldots,m+1\}}|\widetilde{Q}_{j_l}(f_0,\ldots,f_n)|\le h_1\|\fa\|^d.
%$$}
\begin{lemma}\label{lemma3.1}
	For any $D_{j_1},\ldots,D_{j_{m+1}}\in {\mathcal D}$, there exist functions $h_1,h_2\in \mathcal{C}_\mathcal{Q}$ such that
	$$
	h_2\|\fa\|^d\le \max_{l\in\{1,\ldots,m+1\}}|\widetilde{Q}_{j_l}(f_0,\ldots,f_n)|\le h_1\|\fa\|^d.
	$$
\end{lemma}

\begin{proof}[{\bf Proof of Lemma \ref{lemma3.1}}] The second inequality is elementary. Note that
$$
|\widetilde{Q}_{j_l}(f_0,\ldots,f_n)|=\left|\sum_{I\in\I_d}\widetilde{a}_{j_l,I}\fa^I\right|\le \sum_{I\in\I_d}|\widetilde{a}_{j_l,I}|\cdot\|\fa\|^d\ \mbox{for}\ l=1,\ldots,m+1.
$$
Set
$$
h_1:=\sum_{l=1}^{m+1}\sum_{I\in\I_d}|\widetilde{a}_{j_l,I}|.
$$
So we have
$$
\max_{l\in\{1,\ldots,m+1\}}|\widetilde{Q}_{j_l}(f_0,\ldots,f_n)|\le h_1\|\fa\|^d.
$$

For the proof of the first inequality, one can refer to page 10 of \cite{8}. Here, we introduce another proof given in \cite{9-1}.
To do so, we need some results on inertia forms in \cite{10}.

Set
$$
\widetilde{\widetilde{Q}}_{j_l}=\sum_{I\in\I_d}t_{j_l,I}\x^I\in {\mathbb{C}}[{\bf{t}},\x]\ \mbox{with}\ {\bf{t}}=(\ldots,t_{j_l,I},\ldots)\ \mbox{for}\ l=1,\ldots,m+1.
$$
Obviously, $\widetilde{Q}_{j_l}(z)(x_0,\ldots,x_n)=\widetilde{\widetilde{Q}}_{j_l}(\ldots,\widetilde{a}_{j_l,I}(z),\ldots,x_0,\ldots,x_n)$, $l=1,\ldots,m+1$. Let $(\widetilde{\widetilde{Q}}_{j_1},\ldots,\widetilde{\widetilde{Q}}_{j_{m+1}})_{{\mathbb{C}}[{\bf{t}}]}$ be the ideal in the ring of polynomials in $x_0,\ldots,x_n$ with coefficients in ${\mathbb{C}}[{\bf{t}}]$ generated by $\widetilde{\widetilde{Q}}_{j_1},\ldots,\widetilde{\widetilde{Q}}_{j_{m+1}}$. A polynomial $\widetilde{\widetilde{R}}$ in ${\mathbb{C}}[{\bf{t}}]$ is called an inertia form of the polynomials $\widetilde{\widetilde{Q}}_{j_1},\ldots,\widetilde{\widetilde{Q}}_{j_{m+1}}$ if it has the property:
\begin{eqnarray}
x_i^s\cdot\widetilde{\widetilde{R}}\in (\widetilde{\widetilde{Q}}_{j_1},\ldots,\widetilde{\widetilde{Q}}_{j_{m+1}})_{{\mathbb{C}}[{\bf{t}}]} \label{eq2.1}
\end{eqnarray}
for $i=0,\ldots,n$ and for some nonnegative integer $s$. It follows from the definition that the inertia forms of $\widetilde{\widetilde{Q}}_{j_1},\ldots,\widetilde{\widetilde{Q}}_{j_{m+1}}$ form an ideal in ${\mathbb{C}}[{\bf{t}}]$. There exists an inertia form $\widetilde{\widetilde{R}}$ with $\widetilde{\widetilde{R}}(\ldots,t^0_{j,I},\ldots)\neq 0$ if and only if $\widetilde{\widetilde{Q}}_{j_l}(\ldots,t^0_{j_l,I},\ldots,x_0,\ldots,x_n)$, $l=1,\ldots,m+1$, have no common nontrivial solutions in $x_0,\ldots,x_n$ (for special values $t^0_{j_l,I}$).

Take $z_0\in {\mathbb{C}}$ such that all coefficient functions of $\widetilde{Q}_{j_l}(z)$, $l=1,\ldots,m+1$, are holomorphic at $z_0$ and the system of equations
$$
\widetilde{Q}_{j_l}(z_0)(x_0,\ldots,x_n)=0,\ l=1,\ldots,m+1,
$$
have no common nontrivial solutions. Set $t^0_{j_l,I}=\widetilde{a}_{j_l,I}(z_0)$ and the inertia form $\widetilde{\widetilde{R}}$ with $\widetilde{\widetilde{R}}(\ldots,\widetilde{a}_{j_l,I}(z_0),\ldots)\neq 0$. Denote by $\widetilde{R}(z):=\widetilde{\widetilde{R}}(\ldots,\widetilde{a}_{j_l,I},\ldots)\in \mathcal{K}_{\mathcal{Q}}$, then $\widetilde{R}(z)\not\equiv 0$ by $\widetilde{R}(z_0)=\widetilde{\widetilde{R}}(\ldots,\widetilde{a}_{j_l,I}(z_0),\ldots)\neq 0$. From (\ref{eq2.1}), we know that
$$
x_i^s\cdot{\widetilde{R}}\in ({\widetilde{Q}}_{j_1},\ldots,{\widetilde{Q}}_{j_{m+1}})_{\mathcal{K}_{\mathcal{Q}}}\ \mbox{for}\ i=0,\ldots,n,
$$
which implies
\begin{eqnarray}
x_i^s\cdot{\widetilde{R}}=\sum_{l=1}^{m+1}b_{il} \widetilde{Q}_{j_l}, \label{eq2.2}
\end{eqnarray}
where
\begin{eqnarray}
b_{il}=\sum_{I\in\I_{s-d}}\gamma_{il,I}\x^I \label{eq2.3}
\end{eqnarray}
with $\gamma_{il,I}\in \mathcal{K}_{\mathcal{Q}}$.

Now, we continue the proof of the first inequality.

By (\ref{eq2.2}) and (\ref{eq2.3}), we have
\begin{eqnarray*}
|f_i^s\cdot{\widetilde{R}}|&=&\left|\sum_{l=1}^{m+1}\left(\sum_{I\in\I_{s-d}}\gamma_{il,I}\fa^I\right) \widetilde{Q}_{j_l}(f_0,\ldots,f_n)\right|\\
&\le&\sum_{1\le l\le m+1,I\in\I_{s-d}}|\gamma_{il,I}|\cdot\|\fa\|^{s-d}\cdot \max_{l\in\{1,\ldots,m+1\}}|\widetilde{Q}_{j_l}(f_0,\ldots,f_n)|,\ i=0,\ldots,n,
\end{eqnarray*}
i.e.,
\begin{eqnarray*}
\frac{|f_i^s|}{\|\fa\|^{s-d}}\le \sum_{1\le l\le m+1,I\in\I_{s-d}}\left|\frac{\gamma_{il,I}}{{\widetilde{R}}}\right|\cdot\max_{l\in\{1,\ldots,m+1\}}|\widetilde{Q}_{j_l}(f_0,\ldots,f_n)|\ \mbox{for\ all}\ i=0,\ldots,n.
\end{eqnarray*}
Take $i$ such that $\|\fa\|=|f_i|$ and set $h_2=\frac{1}{\sum\limits_{i=0}^n\sum\limits_{1\le l\le m+1,I\in\I_{s-d}}\left|\frac{\gamma_{il,I}}{{\widetilde{R}}}\right|}$, we obtain
$$
h_2\|\fa\|^d\le \max_{l\in\{1,\ldots,m+1\}}|\widetilde{Q}_{j_l}(f_0,\ldots,f_n)|.
$$
Lemma \ref{lemma3.1} is thus proved.
\end{proof}

For each given $z\in \mathbb{C}$ (excluding all zeros and poles of $\widetilde{Q}_j(\fa)$), there exists a renumbering
$\{1(z),\ldots,q(z)\}$ of the indices $\{1,\ldots,q\}$ such that
$$
|\widetilde{Q}_{1(z)}(\fa)(z)|\le|\widetilde{Q}_{2(z)}(\fa)(z)|\le\cdots\le|\widetilde{Q}_{q(z)}(\fa)(z)|.
$$
By Lemma \ref{lemma3.1}, we have $\max\limits_{j\in\{1,\ldots,m+1\}}|\widetilde{Q}_{j(z)}(\fa)(z)|=|\widetilde{Q}_{m+1(z)}(\fa)(z)|\ge h\|\fa(z)\|^d$ for some $h\in \mathcal{C}_\mathcal{Q}$, i.e.,
$$
\frac{\|\fa(z)\|^d}{|\widetilde{Q}_{m+1(z)}(\fa)(z)|}\le \frac{1}{h}.
$$
Hence,
$$
\prod_{j=1}^q\frac{\|\fa(z)\|^d}{|\widetilde{Q}_{j}(\fa)(z)|}\le \frac{1}{h^{q-m}}\prod_{j=1}^m\frac{\|\fa(z)\|^d}{|\widetilde{Q}_{j(z)}(\fa)(z)|}
$$
and
\begin{eqnarray}
\sum_{j=1}^q\lambda_{D_j}(\fa(z))\le \sum_{j=1}^m\lambda_{D_{j(z)}}(\fa(z))+\log h'\ \mbox{with}\ h'\in \mathcal{C}_\mathcal{Q}.
\label{eq2.4}
\end{eqnarray}

If $f$ is algebraically nondegenerate over $\mathcal{K}_{\mathcal{Q}}$, then (\ref{eq1.1}) follows directly from Theorem \ref{thmd}.

If $f$ is algebraically degenerate over $\mathcal{K}_{\mathcal{Q}}$, there is a nonzero homogeneous polynomial $P\in\mathcal{K}_{\mathcal{Q}}[x_0,\ldots,x_n]$ such that $P(f_0,\ldots,f_n)\equiv 0$. We can construct a homogeneous ideal $I_{\mathcal{K}_{\mathcal{Q}}}\subset\mathcal{K}_{\mathcal{Q}}[x_0,\ldots,x_n]$ generated by all homogeneous polynomials $P\in\mathcal{K}_{\mathcal{Q}}[x_0,\ldots,x_n]$ such that $P(f_0,\ldots,f_n)\equiv 0$. Obviously, for all $P\in I_{\mathcal{K}_{\mathcal{Q}}}$, we have $P(f_0,\ldots,f_n)\equiv 0$. Since $\mathcal{K}_{\mathcal{Q}}[x_0,\ldots,x_n]$ is a Noetherian ring, $I_{\mathcal{K}_{\mathcal{Q}}}$ is finitely generated. Assume that $I_{\mathcal{K}_{\mathcal{Q}}}$ is generated by homogeneous polynomials $P_1,\ldots,P_s$. We now use the language of the algebraic geometry on an arbitrary base field and the coordinates of the points are taken over the universal field over the base field. Let ${\mathcal{K}_{\mathcal{Q}}}$ be the base field and $\Omega$ be the universal field over ${\mathcal{K}_{\mathcal{Q}}}$. Consider the projective variety $V\subset\mathbb{P}^n(\Omega)$ (with the base field ${\mathcal{K}_{\mathcal{Q}}}$) constructed by $I_{\mathcal{K}_\mathcal{Q}}$. Let $\deg V=\Delta$ and $\dim V=\ell$ where $0\le \ell\le n$. We first show that $\ell>0$. Otherwise, if we assume that  $f_0\not\equiv 0$, then $(\frac{f_1}{f_0},\ldots,\frac{f_n}{f_0})$ lies in a zero-dimensional affine variety in the affine space $\mathbb{A}^n(\Omega)$. It follows that each coordinate $\frac{f_j}{f_0}$ is algebraic over ${\mathcal{K}_{\mathcal{Q}}}$ (see 16.4 in \cite{10-3}), i.e., $\frac{f_j}{f_0}$ satisfies some algebraic equation $A_t\xi^t+A_{t-1}\xi^{t-1}+\cdots+A_0\equiv 0$ with $A_0,\ldots,A_t\in {\mathcal{K}_{\mathcal{Q}}}$. By Valiron's estimate (see \cite{10-4}), $T_{\frac{f_j}{f_0}}(r)\le \sum\limits_{v=0}^tT_{A_v}(r)+O(1)=o(T_f(r))$, which is a contradiction. Hence, $\ell>0$. Note that $f$ has the following property: there is no homogeneous polynomial $P\in\mathcal{K}_{\mathcal{Q}}[x_0,\ldots,x_n]\setminus I_{\mathcal{K}_{\mathcal{Q}}}$ such that $P(f_0,\ldots,f_n)\equiv 0$.

For a positive integer $N$, let $\mathcal{K}_{\mathcal{Q}}[x_0,\ldots,x_n]_N$ be
the vector space of homogeneous polynomials of degree $N$, and let $I_{\mathcal{K}_{\mathcal{Q}},N}:=I_{\mathcal{K}_{\mathcal{Q}}}\cap\mathcal{K}_{\mathcal{Q}}[x_0,\ldots,x_n]_N$. Denote by $W_N:=\mathcal{K}_{\mathcal{Q}}[x_0,\ldots,x_n]_N/I_{\mathcal{K}_{\mathcal{Q}},N}$. For any $g\in \mathcal{K}_{\mathcal{Q}}[x_0,\ldots,x_n]_N$, let $[g]$ be the projection of $g$ in $W_N$. We refine the basic fact of Hilbert polynomials listed in section \ref{algg} as the following lemma.

\begin{lemma}\label{lemma3.2}
	There exists a positive integer
$N_0$ such that
$$M:=
	\dim_{\mathcal{K}_{\mathcal{Q}}}W_N=\frac{\Delta N^\ell}{\ell!}+O(N^{\ell-1})$$
is a polynomial of $N$ for
$N\geq N_0$.
\end{lemma}

Let $a$ be an arbitrary point in $\mathbb{C}$ such that all coefficients of $P_1,\ldots,P_s$ are holomorphic at $a$. Denote by $I(a)$ the homogeneous ideal in $\mathbb{C}[x_0,\ldots,x_n]$ generated by $P_1(a),\ldots,P_s(a)$. Let $V(a)$ be the variety in $\mathbb{P}^n(\mathbb{C})$ defined by $I(a)$. Then we have the following fact.

\begin{lemma}\label{lemma3.3}
	$\dim V(a)=\ell$ for all $a\in\mathbb{C}$ excluding a discrete subset.
\end{lemma}

To prove Lemma \ref{lemma3.3}, we need the following lemma shown in \cite{9-1} whose proof is included for the sake of completeness.

\begin{lemma}\label{lemma3.4}
	Let $Y$ be a $\mathcal{K}_\mathcal{Q}$-vector subspace in $\mathcal{K}_\mathcal{Q}[x_0,\ldots,x_n]_N$. For each $a\in \mathbb{C}$,
	$$
	Y(a):=\{P(a)(x_0,\ldots,x_n)\mid P\in Y\  \text{with all coefficients holomorphic at}\ a\}
	$$
	is a $\mathbb{C}$-vector subspace of $\mathbb{C}[x_0,\ldots,x_n]_N$. Assume that $\{h_l\}_{l=1}^L$ is a basis of $Y$. Then $\{h_l(a)\}_{l=1}^L$ is a basis of $Y(a)$ for all $a\in\mathbb{C}$ excluding a discrete subset.
\end{lemma}

\begin{proof}[{\bf Proof of Lemma \ref{lemma3.4}}]
Let $(c_{jl})$ be the matrix of coefficients of $\{h_l\}_{l=1}^L$. Since $\{h_l\}_{l=1}^L$ are linearly independent over $\mathcal{K}_\mathcal{Q}$, there exists a square submatrix $A$ of $(c_{jl})$ of order $L$ such that $\det A\not\equiv 0$.

For all $a\in \mathbb{C}$ excluding a discrete subset, we have $\det A(a)\neq 0$ and all coefficients of $\{h_l\}_{l=1}^L$ are holomorphic at $a$. Hence, $\{h_l(a)\}_{l=1}^L\subset Y(a)$ and $\{h_l(a)\}_{l=1}^L$ are linearly independent over $\mathbb{C}$. On the other hand, for any $P(a)\in Y(a)$, there exists $P\in Y$ whose coefficients are holomorphic at $a$. Since $\{h_l\}_{l=1}^L$ is a basis of $Y$, there exist meromorphic functions $t_l\in \mathcal{K}_\mathcal{Q}$ such that $P=\sum_{l=1}^Lt_lh_l$. Now, we prove that $t_l$ are also holomorphic at $a$. In fact, there are coefficients $b_l$, $l=1,\ldots,L$, of $P$ such that the following system of linear equations
$$
A\left(\!\begin{array}{c}t_1\\\vdots\\
t_L
\end{array}\!\right)=\left(\!\begin{array}{c}b_1\\\vdots\\
b_L
\end{array}\!\right)
$$
has the unique solution $(t_1,\ldots,t_L)$. We note that $\det A(a)\neq 0$ and $b_1,\ldots,b_L$ are holomorphic at $a$. Hence, $t_1,\ldots,t_L$ are holomorphic at $a$. We have $P(a)=\sum_{l=1}^Lt_l(a)h_l(a)$, $t_l(a)\in \mathbb{C}$. Therefore $\{h_l(a)\}_{l=1}^L$ is a basis of $Y(a)$ and in particular $\dim_{\mathcal{K}_\mathcal{Q}}Y=\dim_{\mathbb{C}}Y(a)$.
\end{proof}

\begin{proof}[{\bf Proof of Lemma \ref{lemma3.3}}]
	Let $a$ be an arbitrary point in $\mathbb{C}$ such that all coefficients of $P_1,\ldots,P_s$ are holomorphic at $a$ (this is true for all points in $\mathbb{C}$ excluding a discrete subset). Denote by %$I_{\mathcal{K}_\mathcal{Q}}(V)_N:=I_{\mathcal{K}_\mathcal{Q}}\cap \mathcal{K}_\mathcal{Q}[x_0,\ldots,x_n]_N$,
	 $I(a)_N:=I(a)\cap\mathbb{C}[x_0,\ldots,x_n]_N$, and define
	$$I_{\mathcal{K}_\mathcal{Q},N}(a):=\{P(a)(x_0,\ldots,x_n)\mid P\in I_{\mathcal{K}_\mathcal{Q},N}\  \text{with all coefficients holomorphic at}\ a\}.$$
	
\begin{claim}\label{claim}
	 $I(a)_N=I_{\mathcal{K}_\mathcal{Q},N}(a)$ for all $a\in \mathbb{C}$ excluding a discrete subset.
\end{claim}

	On the one hand, for all $P\in I(a)_N$, i.e., $P=\sum\limits_{j=1}^s R_j P_j(a)$ with $R_j\in \mathbb{C}[x_0,\ldots,x_n]$, let $\widetilde{P}=\sum\limits_{j=1}^s R_j P_j\in I_{\mathcal{K}_\mathcal{Q},N}$, then the coefficients of $\widetilde{P}$ are holomorphic at $a$ and $\widetilde{P}(a)=P$, so $P\in I_{\mathcal{K}_\mathcal{Q},N}(a)$ which implies $I(a)_N\subset I_{\mathcal{K}_\mathcal{Q},N}(a)$.\par
	On the other hand, assume that $\{h_l\}_{l=1}^L$ is a basis of  $I_{\mathcal{K}_\mathcal{Q},N}$, let $h_l=\sum\limits_{j=1}^s R_{jl} P_j$ with $R_{jl}\in \mathcal{K}_\mathcal{Q}[x_0,\ldots,x_n]$, since  the coefficients of $R_{jl}$ are holomorphic at all points of $\mathbb{C}$ excluding a discrete subset, thus $\{h_l(a)\}_{l=1}^L\subset I(a)_N$ for all points $a\in \mathbb{C}$ excluding a discrete set. By Lemma \ref{lemma3.4}, $\{h_l(a)\}_{l=1}^L$ is a basis of $I_{\mathcal{K}_\mathcal{Q},N}(a)$ for all $a\in\mathbb{C}$ excluding a discrete subset. From the above inclusion relation ``$I(a)_N\subset I_{\mathcal{K}_\mathcal{Q},N}(a)$'', we know that $\{h_l(a)\}_{l=1}^L$ is a basis of $I(a)_N$ for all $a\in\mathbb{C}$ excluding a (maybe larger) discrete subset, which completes the proof of the claim.\par
	%From the claim, we know that, for such $a\in\mathbb{C}$,
	Combining the claim and Lemma \ref{lemma3.4}, we have
	\begin{equation}\label{claim2.3}
	\dim_\mathbb{C} I(a)_N=\dim_\mathbb{C} I_{\mathcal{K}_\mathcal{Q},N}(a)=\dim_{\mathcal{K}_\mathcal{Q}}I_{\mathcal{K}_\mathcal{Q},N}
	\end{equation}for all  $a\in \mathbb{C}$ excluding a discrete subset. For such an $a\in\mathbb{C}$, by (\ref{claim2.3}) and Lemma \ref{lemma3.2},
	\begin{eqnarray*}
	\dim_\mathbb{C}\frac{\mathbb{C}[x_0,\ldots,x_n]_N}{I(a)_N}
	&=&\dim_\mathbb{C}\frac{\mathbb{C}[x_0,\ldots,x_n]_N}{I_{\mathcal{K}_\mathcal{Q},N}(a)} =\dim_\mathbb{\mathcal{K}_\mathcal{Q}}\frac{\mathcal{K}_\mathcal{Q}[x_0,\ldots,x_n]_N}{I_{\mathcal{K}_\mathcal{Q},N}}\\
&=&\frac{\Delta N^\ell}{\ell!}+O(N^{\ell-1})\ \text{for}\ N\gg 0.
	\end{eqnarray*}
	By the theory of Hilbert polynomials, we know that $\dim V(a)=\ell$, which completes the proof.
\end{proof}

Consider an arbitrary subset of $\{\widetilde{Q}_1,\ldots,\widetilde{Q}_q\}$ of $m+1$ polynomials (e.g., $\{\widetilde{Q}_1,\ldots,\widetilde{Q}_{m+1}\}$), then $\widetilde{Q}_1,\ldots,\widetilde{Q}_{m+1}$ are in $m$-subgeneral position on $\mathbb{P}^n(\mathbb{C})$.

\begin{lemma}\label{lemma3.5}
For a fixed point $a\in\mathbb{C}$ satisfying the following conditions\par
%\begin{enumerate}[(i)]
%	\item the coefficients of  $P_1,\ldots,P_s,\widetilde{Q}_1,\ldots,\widetilde{Q}_{m+1}$ are holomorphic at $a$,
%	\item $\widetilde{Q}_1(a),\ldots,\widetilde{Q}_{m+1}(a)$ have no non-trivial common zeros,
%	\item $\dim V(a)=\ell$,
%\end{enumerate}
\textit{(i) the coefficients of  $P_1,\ldots,P_s,\widetilde{Q}_1,\ldots,\widetilde{Q}_{m+1}$ are holomorphic at $a$,}\par
\textit{(ii) $\widetilde{Q}_1(a),\ldots,\widetilde{Q}_{m+1}(a)$ have no nontrivial common zeros,}\par
\textit{(iii) $\dim V(a)=\ell$,} \par
\textit{there exist polynomials $\widetilde{P}_1(a)=\widetilde{Q}_1(a),\widetilde{P}_2(a),\ldots,\widetilde{P}_{\ell+1}(a)\in\mathbb{C}[x_0,\ldots,x_n]$ with
\begin{equation*}
\widetilde{P}_t(a)=\sum\limits_{j=2}^{m-\ell+t}c_{tj}\widetilde{Q}_j(a),\quad c_{tj}\in\mathbb{C},\ t\geq 2,
\end{equation*} such that
\begin{equation*}
\left(\bigcap\limits_{t=1}^{\ell+1}\widetilde{P}_t(a)\right)\cap V(a)=\emptyset.
\end{equation*}}
\end{lemma}

(We note that the proof of Lemma \ref{lemma3.5} is similar to that of Lemma 3.1 in \cite{15}, which is omitted here. Moreover, (i)--(iii) are satisfied for all $a\in\mathbb{C}$ excluding a discrete subset.)
%\begin{proof}[{\bf Proof of Lemma 2.4.}]
%By using the method of \cite{15} (Lemma 3.1 in \cite{15}), we can construct such homogenous polynomials.
%\end{proof}

Let $\widetilde{P}_1=\widetilde{Q}_1,\ \widetilde{P}_t=\sum\limits_{j=2}^{m-\ell+t}c_{tj}\widetilde{Q}_j,\ t=2,\ldots,\ell+1$, which are homogeneous polynomials in $\mathcal{K}_\mathcal{Q}[x_0,\ldots,x_n]$ , then
\begin{equation}\label{eqNN-1}
\left\{
\begin{array}{ccc}
{P_1}&=&0 \\ &\vdots& \\ P_s&=&0 \\ \widetilde{P}_1&=&0 \\ &\vdots& \\ \widetilde{P}_{\ell+1}&=&0
\end{array}
\right.
\end{equation}have no  nontrivial common zeros for all $a\in\mathbb{C}$ excluding a discrete set.

Since there are only finitely many choices of $m+1$ polynomials from $\{\widetilde{Q}_1,\ldots,\widetilde{Q}_{q}\}$, the total number of such $\widetilde{P}_j$'s is finite, so there exists a constant $C>0$, for $t=2,\ldots,\ell$ and all $z\in\mathbb{C}$ (excluding all zeros and poles of $\widetilde{Q}_j(\fa)$), we can construct $\widetilde{P}_{1(z)},\ldots,\widetilde{P}_{(\ell+1)(z)}$ from $\widetilde{Q}_{1(z)},\ldots,\widetilde{Q}_{(m+1)(z)}$ such that
\begin{equation*}
|\widetilde{P}_{t(z)}(\fa(z))|\leq C\max\limits_{2\leq j\leq m-\ell+t}|\widetilde{Q}_{j(z)}(\fa(z))|=C|\widetilde{Q}_{(m-\ell+t)(z)}(\fa(z))|\ \text{for}\ 2\leq t\leq \ell,
\end{equation*}i.e.,
\begin{equation*}
\lambda_{D_{(m-\ell+t)(z)}}(\fa(z))\leq\lambda_{\widetilde{P}_{t(z)}}(\fa(z))+\log h''\ \text{with}\  h''\in\mathcal{C}_\mathcal{Q}\ \text{for}\ 2\leq t\leq \ell.
\end{equation*}

Combing with \eqref{eq2.4}, we have
\begin{align}
\sum\limits_{j=1}^q\lambda_{D_j}(\fa(z))&\ \leq\  \sum\limits_{t=1}^\ell\lambda_{\widetilde{P}_t(z)}(\fa(z))+\sum\limits_{j=2}^{m-\ell+1}\lambda_{D_{j(z)}}(\fa(z))+\log h'''\nonumber\\
&\ \leq\ (m-\ell+1)\sum\limits_{t=1}^\ell\lambda_{\widetilde{P}_t(z)}(\fa(z))+\log h''',\quad h'''\in\mathcal{C}_\mathcal{Q}.\label{eq2.5-0}
\end{align}

Fix a basis $[\phi_1],\ldots,[\phi_M]$ of $W_N$ with $\phi_1,\ldots,\phi_M\in\mathcal{K}_\mathcal{Q}[x_0,\ldots,x_n]$, and let
$$
F:{\mathbb{C}}\rightarrow
{\mathbb{P}}^{M-1}({\mathbb{C}})
$$
be a holomorphic map with a reduced representation ${\bf F}=(\phi_1(\fa),\ldots,\phi_M(\fa))$.
Since $\fa$ satisfies $P(\fa)\not\equiv 0$ for all homogeneous polynomials $P\in\mathcal{K}_{\mathcal{Q}}[x_0,\ldots,x_n]\setminus I_{\mathcal{K}_{\mathcal{Q}}}$, $F$ is linearly nondegenerate over $\mathcal{K}_{\mathcal{Q}}$. We have
\begin{eqnarray}
T_F(r)=NT_f(r)+o(T_f(r)).\label{eq2.5}
\end{eqnarray}

For every positive integer $N$ with $d|N$, we use the
following filtration of the vector space $W_N$ with respect to $\widetilde{P}_{1(z)},\ldots,\widetilde{P}_{\ell(z)}$. This filtration is a generalization of Corvaja-Zannier's filtration (in \cite{11,12}), which is given in \cite{9-1}.

Arrange, by the lexicographic order, the $\ell$-tuples ${\bf i}=(i_1,\ldots,i_\ell)$ of
nonnegative integers and set $\|{\bf i}\|=\sum_ji_j$.

\begin{defi}\label{def3.1}
 (i) For each ${\bf i}\in \mathbb{Z}_{\ge 0}^\ell$ and nonnegative integer $N$ with $N\ge d\|{\bf i}\|$, denote by $I_N^{\bf i}$ the subspace of ${\mathcal{K}_\mathcal{Q}}[x_0,\ldots,x_n]_{N-d\|{\bf i}\|}$ consisting of all $\gamma\in {\mathcal{K}_\mathcal{Q}}[x_0,\ldots,x_n]_{N-d\|{\bf i}\|}$ such that
	$$
	\widetilde{P}_{1(z)}^{i_1}\cdots \widetilde{P}_{\ell(z)}^{i_\ell}\gamma-\sum_{{\bf e}=(e_1,\ldots,e_\ell)>{\bf i}}\widetilde{P}_{1(z)}^{e_1}\cdots \widetilde{P}_{\ell(z)}^{e_\ell}\gamma_{\bf e}\in I_{\mathcal{K}_\mathcal{Q},N}
	$$
	(or
	$[\widetilde{P}_{1(z)}^{i_1}\cdots \widetilde{P}_{\ell(z)}^{i_\ell}\gamma]=[\sum_{{\bf e}>{\bf i}}\widetilde{P}_{1(z)}^{e_1}\cdots \widetilde{P}_{\ell(z)}^{e_\ell}\gamma_{\bf e}]$ on $W_{N}$)
	for some $\gamma_{\bf e}\in {\mathcal{K}_\mathcal{Q}}[x_0,\ldots,x_n]_{N-d\|{\bf e}\|}$.
	
(ii) Denote by $I^{\bf i}$ the homogeneous ideal in ${\mathcal{K}_\mathcal{Q}}[x_0,\ldots,x_n]$ generated by $\bigcup_{N\ge d\|{\bf i}\|}I^{\bf i}_N$.
\end{defi}

\begin{remark}\label{rkN3.3}
From this definition, we have the following properties.	
	
\textit{	(i) $(I_{\mathcal{K}_\mathcal{Q}},\widetilde{P}_{1(z)},\ldots,\widetilde{P}_{\ell(z)})_{N-d\|{\bf i}\|}\subset I^{\bf i}_N\subset {\mathcal{K}_\mathcal{Q}}[x_0,\ldots,x_n]_{N-d\|{\bf i}\|}$, where we denote by $(I_{\mathcal{K}_\mathcal{Q}},\widetilde{P}_{1(z)},\ldots,\widetilde{P}_{\ell(z)})$ the ideal in ${\mathcal{K}_\mathcal{Q}}[x_0,\ldots,x_n]$ generated by $I_{\mathcal{K}_\mathcal{Q}}\cup\{\widetilde{P}_{1(z)},\ldots,\widetilde{P}_{\ell(z)}\}$.}
	
\textit{	(ii) $I^{\bf i}\cap {\mathcal{K}_\mathcal{Q}}[x_0,\ldots,x_n]_{N-d\|{\bf i}\|}=I^{\bf i}_N$.}
	
\textit{    (iii) $\frac{{\mathcal{K}_\mathcal{Q}}[x_0,\ldots,x_n]}{I^{\bf i}}$ is a graded module over ${\mathcal{K}_\mathcal{Q}}[x_0,\ldots,x_n]$.}
	
\textit{	(iv) If ${\bf i}_1-{\bf i}_2:=(i_{1,1}-i_{2,1},\ldots,i_{1,\ell}-i_{2,\ell})\in \mathbb{Z}_{\ge 0}^\ell$, then $I_{N}^{{\bf i}_2}\subset I_{N+d\|{\bf i}_1\|-d\|{\bf i}_2\|}^{{\bf i}_1}$. Hence $I^{{\bf i}_2}\subset I^{{\bf i}_1}$.}
\end{remark}

\begin{lemma}\label{lemma3.6}
	$\{I^{\bf i}|{\bf i}\in \mathbb{Z}_{\ge 0}^\ell\}$ is a finite set.
\end{lemma}

\begin{proof}[{\bf Proof of Lemma \ref{lemma3.6}}]
	Suppose that $\sharp\{I^{\bf i}|{\bf i}\in \mathbb{Z}_{\ge 0}^\ell\}=\infty$. We can construct a sequence $\{{\bf i}_j\}_{j=1}^{\infty}$ such that ${\bf i}_{j+1}-{\bf i}_{j}\in \mathbb{Z}_{\ge 0}^\ell$ and $\{I^{{\bf i}_j}\}_{j=1}^{\infty}$ consisting of pairwise different ideals. By (iv) of Remark \ref{rkN3.3},
	$$
	I^{{\bf i}_{1}}\subset I^{{\bf i}_{2}}\subset \cdots \subset I^{{\bf i}_{j}}\subset I^{{\bf i}_{j+1}}\subset\cdots,
	$$
	which contradicts the fact that ${\mathcal{K}_\mathcal{Q}}[x_0,\ldots,x_n]$ is a Noetherian ring.
\end{proof}

Denote by
\begin{eqnarray}
\Delta_N^{\bf i}:=\dim_{{\mathcal{K}_\mathcal{Q}}}\frac{{\mathcal{K}_\mathcal{Q}}[x_0,\ldots,x_n]_{N-d\|{\bf i}\|}}{I^{\bf i}_{N}}.\label{eqN-D-1}
\end{eqnarray}

\begin{lemma}\label{lemma3.7}
	\textit{(i) There exists a positive integer $N_0$ such that, for each ${\bf i}\in \mathbb{Z}^\ell_{\ge 0}$, $\Delta_N^{\bf i}$ is independent of $N$ for all $N$ satisfying $N-d\|{\bf i}\|>N_0$.}
	
	\textit{(ii) There is an integer $\overline{\Delta}$ such that $\Delta_N^{\bf i}\le \overline{\Delta}$ for all ${\bf i}\in \mathbb{Z}^\ell_{\ge 0}$ and $N$ satisfying $N-d\|{\bf i}\|\ge 0$.}
\end{lemma}

\begin{proof}[{\bf Proof of Lemma \ref{lemma3.7}}]
	(i)  For each ${\bf i}\in \mathbb{Z}_{\ge 0}^{\ell}$, by (iii) of Remark \ref{rkN3.3},
	$$
	\Delta_N^{\bf i}=\dim_{{\mathcal{K}_\mathcal{Q}}}\frac{{\mathcal{K}_\mathcal{Q}}[x_0,\ldots,x_n]_{N-d\|{\bf i}\|}}{I^{\bf i}_{N}}
	$$
	is a polynomial of $N$ for $N$ big enough. (See Theorem 14 in \cite{10-2}.)
	
	Similar to the proof of Lemma \ref{lemma3.3}, for all $a\in\mathbb{C}$ excluding a discrete subset,
	\begin{equation*}
	(I(a),\widetilde{P}_{1(z)}(a),\ldots,\widetilde{P}_{\ell(z)}(a))_{N-d\|{\bf i}\|}=(I_{\mathcal{K}_\mathcal{Q}},\widetilde{P}_{1(z)},\ldots,\widetilde{P}_{\ell(z)})_{N-d\|{\bf i}\|}(a)
	\end{equation*}and
	\begin{align*}
	\dim_\mathbb{C}(I(a),\widetilde{P}_{1(z)}(a),\ldots,\widetilde{P}_{\ell(z)}(a))_{N-d\|{\bf i}\|} & =\dim_\mathbb{C}(I_{\mathcal{K}_\mathcal{Q}},\widetilde{P}_{1(z)},\ldots,\widetilde{P}_{\ell(z)})_{N-d\|{\bf i}\|}(a)\\
	& =\dim_{{\mathcal{K}_\mathcal{Q}}}(I_{\mathcal{K}_\mathcal{Q}},\widetilde{P}_{1(z)},\ldots,\widetilde{P}_{\ell(z)})_{N-d\|{\bf i}\|}.
	\end{align*}
	According to Lemma \ref{lemma3.5} and equation \eqref{eqNN-1},  $\widetilde{P}_{1(z)}(a),\ldots,\widetilde{P}_{\ell(z)}(a)$ are in general position in $V(a)$ for all $a\in\mathbb{C}$ excluding a discrete subset. Thus, we can find a point $a\in\mathbb{C}$ such that
	\begin{align*}
	\dim_{\mathcal{K}_\mathcal{Q}}\frac{\mathcal{K}_\mathcal{Q}[x_0,\ldots,x_n]_{N-d\|{\bf i}\|}}{(I_{\mathcal{K}_\mathcal{Q}},\widetilde{P}_{1(z)},\ldots,\widetilde{P}_{\ell(z)})_{N-d\|{\bf i}\|}} & =\dim_\mathbb{C}\frac{\mathbb{C}[x_0,\ldots,x_n]_{N-d\|{\bf i}\|}}{(I_{\mathcal{K}_\mathcal{Q}},\widetilde{P}_{1(z)},\ldots,\widetilde{P}_{\ell(z)})_{N-d\|{\bf i}\|}(a)}\\
	& =\dim_\mathbb{C}\frac{\mathbb{C}[x_0,\ldots,x_n]_{N-d\|{\bf i}\|}}{(I(a),\widetilde{P}_{1(z)}(a),\ldots,\widetilde{P}_{\ell(z)}(a))_{N-d\|{\bf i}\|}}
	\end{align*}
	and, by the theory of Hilbert functions, there exists an integer $N_1>0$ such that
	\begin{equation*}
	\dim_\mathbb{C}\frac{\mathbb{C}[x_0,\ldots,x_n]_{N-d\|{\bf i}\|}}{(I(a),\widetilde{P}_{1(z)}(a),\ldots,\widetilde{P}_{\ell(z)}(a))_{N-d\|{\bf i}\|}}
	\end{equation*}
	is a constant for all ${\bf i}\in \mathbb{Z}_{\ge 0}^{\ell}$ and $N$ with $N-d\|{\bf i}\|>N_1$.
	%i.e., $\dim_{\mathcal{K}_\mathcal{Q}}\frac{\mathcal{K}_\mathcal{Q}[x_0,\ldots,x_n]_{N-d\|{\bf i}\|}}{(I_{\mathcal{K}_\mathcal{Q}},\widetilde{P}_{1(z)},\ldots,\widetilde{P}_{\ell(z)})_{N-d\|{\bf i}\|}}$ is a constant for all ${\bf i}\in \mathbb{Z}_{\ge 0}^{\ell}$ and $N$ with $N-d\|{\bf i}\|>N_1$.	
	
	By (i) of Remark \ref{rkN3.3},
	\begin{eqnarray}
	\dim_{\mathcal{K}_\mathcal{Q}}\frac{{\mathcal{K}_\mathcal{Q}}[x_0,\ldots,x_n]_{N-d\|{\bf i}\|}}{I^{\bf i}_{N}}\le\dim_{\mathcal{K}_\mathcal{Q}}\frac{{\mathcal{K}_\mathcal{Q}}[x_0,\ldots,x_n]_{N-d\|{\bf i}\|}}{(I_{\mathcal{K}_\mathcal{Q}},\widetilde{P}_{1(z)},\ldots,\widetilde{P}_{\ell(z)})_{N-d\|{\bf i}\|}}.\label{eqN-D-2}
	\end{eqnarray}
	Hence, there is an integer $N_2^{\bf i}(>N_1)$ such that $\Delta_N^{\bf i}$ is also a constant for all $N$ satisfying $N-d\|{\bf i}\|>N_2^{\bf i}$. Set $\Delta^{\bf i}$ to be this constant. We note that $N_2^{\bf i}$ depends on $I^{\bf i}$ and $\{I^{\bf i}|{\bf i}\in \mathbb{Z}_{\ge 0}^\ell\}$ is a finite set by Lemma \ref{lemma3.6}. Take $N_0=\max\{N_2^{\bf i}|{\bf i}\in \mathbb{Z}_{\ge 0}^\ell\}$, we have $\Delta^{\bf i}=\Delta^{\bf i}_N$ for all ${\bf i}\in \mathbb{Z}_{\ge 0}^\ell$ and $N$ satisfying $N-d\|{\bf i}\|>N_0$.
	
	(ii) By (\ref{eqN-D-2}), we have $\Delta^{\bf i}_N\le \dim_{\mathcal{K}_\mathcal{Q}}\frac{{\mathcal{K}_\mathcal{Q}}[x_0,\ldots,x_n]_{N-d\|{\bf i}\|}}{(I_{\mathcal{K}_\mathcal{Q}},\widetilde{P}_{1(z)},\ldots,\widetilde{P}_{\ell(z)})_{N-d\|{\bf i}\|}}$. Hence, taking $$\overline{\Delta}:=\max\left\{\left.\dim_{\mathcal{K}_\mathcal{Q}}\frac{{\mathcal{K}_\mathcal{Q}}[x_0,\ldots,x_n]_{N}}{(I_{\mathcal{K}_\mathcal{Q}},\widetilde{P}_{1(z)},\ldots,\widetilde{P}_{\ell(z)})_{N}}\right|N=0,1,\ldots,N_1+1\right\},$$ we get (ii) of Lemma \ref{lemma3.7}.
\end{proof}

Set $\Delta_0:=\min_{{\bf i}\in \mathbb{Z}_{\ge 0}^\ell}\Delta^{\bf i}$, then $\Delta_0=\Delta^{{\bf i}_0}$
for some ${\bf i}_0\in \mathbb{Z}_{\ge 0}^\ell$.

\begin{remark}\label{rkN3.6}
\textit{	By (iv) of Remark \ref{rkN3.3}, if ${\bf i}-{\bf i}_0\in \mathbb{Z}_{\ge 0}^\ell$, then $\Delta^{\bf i}\le \Delta^{{\bf i}_0}$.}
\end{remark}

Now, for an integer $N$ big enough, divisible by $d$, we construct the following filtration of $W_N$ with respect to $\{\widetilde{P}_{1(z)},\ldots,\widetilde{P}_{\ell(z)}\}$.

Denote by $\tau_N$ the set of ${\bf i}\in \mathbb{Z}_{\ge 0}^\ell$ with $N-d\|{\bf i}\|\ge 0$, arranged by
the lexicographic order.

Define
$$
W_{{\bf i}}=\sum_{{\bf e}\ge{\bf i}}\widetilde{P}_{1(z)}^{e_1}\cdots \widetilde{P}_{\ell(z)}^{e_\ell}\cdot{{\mathcal{K}_\mathcal{Q}}[x_0,\ldots,x_n]}_{N-d\|{\bf e}\|}.
$$
Plainly $W_{(0,\ldots,0)}={\mathcal{K}_\mathcal{Q}}[x_0,\ldots,x_n]_N$ and $W_{{\bf i}}\supset W_{{\bf i}'}$ if
${\bf i}'> {\bf i}$, so $\{W_{{\bf i}}\}$ is a filtration of ${\mathcal{K}_\mathcal{Q}}[x_0,\ldots,x_n]_N$. Set $W^*_{{\bf i}}=\{[g]\in W_N|g\in W_{{\bf i}}\}$. Hence, $\{W^*_{{\bf i}}\}$ is a filtration of $W_N$.

\begin{lemma}\label{lemma3.8}
\textit{	Suppose that ${\bf i}'$ follows ${\bf i}$ in the lexicographic order, then
	$$
	\frac{W_{\bf i}^*}{W_{{\bf i}'}^*}\cong\frac{{\mathcal{K}_\mathcal{Q}}[x_0,\ldots,x_n]_{N-d\|{\bf i}\|}}{I^{\bf i}_{N}}.
	$$}
\end{lemma}
\begin{proof}[{\bf Proof of Lemma \ref{lemma3.8}}]
	Define a vector space homomorphism
	$$
	\varphi:{\mathcal{K}_\mathcal{Q}}[x_0,\ldots,x_n]_{N-d\|{\bf i}\|}\rightarrow \frac{W_{\bf i}^*}{W_{{\bf i}'}^*},
	$$
	which maps $\gamma\in {\mathcal{K}_\mathcal{Q}}[x_0,\ldots,x_n]_{N-d\|{\bf i}\|}$ to $[\widetilde{P}_{1(z)}^{i_1}\cdots \widetilde{P}_{\ell(z)}^{i_\ell}\gamma](\in W_{\bf i}^*)$ modulo $W_{{\bf i}'}^*$. Obviously, it is surjective.
	
	Let $\ker \varphi$ be the kernel of $\varphi$. Suppose $\gamma\in \ker \varphi$. This means
	$$
	[\widetilde{P}_{1(z)}^{i_1}\cdots \widetilde{P}_{\ell(z)}^{i_\ell}\gamma]\in W_{{\bf i}'}^*
	$$
	(or
	$[\widetilde{P}_{1(z)}^{i_1}\cdots \widetilde{P}_{\ell(z)}^{i_\ell}\gamma]=[\sum_{{\bf e}>{\bf i}}\widetilde{P}_{1(z)}^{e_1}\cdots \widetilde{P}_{\ell(z)}^{e_\ell}\gamma_{\bf e}]$ for some $\gamma_{\bf e}\in {\mathcal{K}_\mathcal{Q}}[x_0,\ldots,x_n]_{N-d\|{\bf e}\|}$), i.e., $\gamma\in I_N^{\bf i}$. Hence, $\ker \varphi\subset I_N^{\bf i}$. On the other hand, if $\gamma\in I_N^{\bf i}$, then, there exist $\gamma_{\bf e}\in {\mathcal{K}_\mathcal{Q}}[x_0,\ldots,x_n]_{N-d\|{\bf e}\|}$ such that
	$$
	[\widetilde{P}_{1(z)}^{i_1}\cdots \widetilde{P}_{\ell(z)}^{i_\ell}\gamma]=[\sum_{{\bf e}>{\bf i}}\widetilde{P}_{1(z)}^{e_1}\cdots \widetilde{P}_{\ell(z)}^{e_\ell}\gamma_{\bf e}]\in W_{\bf i}^*,
	$$
	i.e., $\gamma\in \ker \varphi$. Hence, $\ker \varphi= I_N^{\bf i}$, which completes the proof of Lemma \ref{lemma3.8}.
\end{proof}

Combining with (\ref{eqN-D-1}), we have
$$
\dim{\frac{W_{{\bf i}}^*}{W_{{\bf i}'}^*}}=\Delta^{{\bf i}}_N.
$$

Set
$$
\tau^0_N=\{{\bf i}\in \tau_N|N-d\|{\bf i}\|>N_0\ \mbox{and}\ {\bf i}-{\bf i}_0\in\mathbb{Z}_{\ge 0}^\ell\}.
$$
We have the following properties.

\begin{lemma}\label{lemma3.9}
\textit{	(i) $\Delta_0=\Delta^{\bf i}$ for all ${\bf i}\in \tau^0_N$.}
	
\textit{	(ii) $\sharp\tau^0_N=\frac{1}{d^{\ell}}\frac{N^{\ell}}{\ell!}+O(N^{\ell-1})$.}
	
\textit{	(iii) $\Delta^{\bf i}_N=\Delta d^{\ell}$ for all ${\bf i}\in \tau^0_N$.}
\end{lemma}

\begin{proof}[{\bf Proof of Lemma \ref{lemma3.9}}]
	(i) By the definition of $\tau^0_N$, we have $\Delta^{\bf i}_N=\Delta^{\bf i}$ for ${\bf i}\in \tau^0_N$. On the other hand, $\Delta^{\bf i}\le \Delta^{{\bf i}_0}$ (note that ${\bf i}-{\bf i}_0\in \mathbb{Z}_{\ge 0}^\ell$ and Remark \ref{rkN3.6}). By the minimality of $\Delta^{{\bf i}_0}$, we obtain (i).
	
	(ii) Clearly,
	$$
	\sharp\tau_N=\left(\begin{matrix}\frac{N}{d}+\ell \\ \ell\end{matrix}\right)=\frac{1}{d^{\ell}}\frac{N^{\ell}}{\ell!}+O(N^{\ell-1}),\ \ \
	\sharp\{{\bf i}\in\tau_N|N-d\|{\bf i}\|\le N_0\}=O(N^{\ell-1})
	$$
	and
	$$
	\sharp\{{\bf i}\in\tau_N|{\bf i}-{\bf i}_0=(i_1-i_{0,1},\ldots,i_\ell-i_{0,\ell})\ \mbox{with\ some}\ i_j-i_{0,j}<0\}=O(N^{\ell-1}).
	$$
	It implies that $\sharp\tau^0_N=\frac{1}{d^{\ell}}\frac{N^{\ell}}{\ell!}+O(N^{\ell-1})$.
	
	(iii) By (ii) of Lemma \ref{lemma3.7}, $\Delta_N^{{\bf i}}$ is bounded for all ${\bf i}$ and $N$. Hence, combining (i), (ii) and Lemma \ref{lemma3.2},
	\begin{eqnarray*}
		\frac{\Delta N^\ell}{\ell!}+O(N^{\ell-1})&=&\sum_{{\bf i}\in \tau_N}\Delta_N^{\bf i}=\Delta_0\cdot \sharp\tau^0_N+\sum_{{\bf i}\in \tau_N\setminus\tau^0_N}\Delta_N^{\bf i}\\
		&=&\Delta_0\left(\frac{1}{d^{\ell}}\frac{N^\ell}{\ell!}+O(N^{\ell-1})\right)+O(N^{\ell-1}),
	\end{eqnarray*}
	which implies $\Delta_0=\Delta d^{\ell}$.
\end{proof}

We choose a basis $\mathcal{B}=\{[\psi_1],\ldots,[\psi_M]\}$ of $W_N$ with respect to the above filtration. Let $[\psi]$ be an element of the basis, which lies in
$W^*_{{\bf i}}/W^*_{{\bf i}'}$, we may write $\psi=\widetilde{P}_{1(z)}^{i_1}\cdots \widetilde{P}_{\ell(z)}^{i_\ell}\gamma$, where $\gamma\in {\mathcal{K}_\mathcal{Q}}[x_0,\ldots,x_n]_{N-d\|{\bf i}\|}$. For every $1\le j\le \ell$, we have
\begin{eqnarray}
\sum_{{\bf i}\in \tau_N}\Delta_N^{\bf i}i_j=\frac{\Delta N^{\ell+1}}{(\ell+1)!d}+O(N^{\ell}).\label{eq3-17}
\end{eqnarray}
(The proof of (\ref{eq3-17}) is similar to (3.6) in \cite{12}). Hence
\begin{equation}\label{eq2.12}
\sum_{t=1}^{M}\lambda_{\psi_t}(\fa(z))\ge \left( \frac{\Delta N^{\ell+1}}{(\ell+1)!d}+O(N^{\ell})\right)\cdot \sum_{j=1}^\ell\lambda_{\widetilde{P}_{j(z)}}(\fa(z))+\log h'''',\quad h''''\in\mathcal{C}_\mathcal{Q}.
\end{equation}

The basis $[\psi_1],\ldots,[\psi_M]$ can be written as linear forms $L_1,\ldots,L_M$ (over $\mathcal{K}_\mathcal{Q}$) in the basis $[\phi_1],\ldots,[\phi_M]$  and $\psi_t(\fa)=L_t({\bf{F}})$. Since there are only finitely many choices of
$\{\widetilde{Q}_{1(z)},\ldots,\widetilde{Q}_{(m+1)(z)}\}$, the collection of all possible linear forms $L_t(1\leq t\leq M)$ is a finite set, and denote it by $\mathcal{L}=\{L_\mu\}_{\mu=1}^\Lambda,\Lambda<\infty$. It is easy to see that $\mathcal{K}_\mathcal{L}\subset\mathcal{K}_\mathcal{Q}$.

By \eqref{eq2.5-0} and \eqref{eq2.12}, take integration on the circle of radius $r$, we have
\begin{eqnarray}
&&\|\ \frac{\Delta N^{\ell+1}}{(\ell+1)!d}\left(1+o(1)\right)\cdot\sum_{j=1}^qm_f(r,D_j)\nonumber\\
&\le& (m-\ell+1)\int_0^{2\pi}\max_{B}\sum_{j\in B}\lambda_{L{j}}({\bf{F}}(re^{i\theta}))\frac{d\theta}{2\pi}+o(T_f(r)),\label{eq3-22}
\end{eqnarray}
where the set $B$ ranges over all subset of $\{1,\ldots,\Lambda\}$ such that the linear forms $\{L_j\}_{j\in B}$ are linearly independent. By Theorem A4.2.1 in \cite{13}, we have, for any $\epsilon>0$,
\begin{eqnarray}
\|\int_0^{2\pi}\max_{B}\sum_{j\in
B}\lambda_{L_j}({\bf{F}}(re^{i\theta}))\frac{d\theta}{2\pi}
\le (M+\epsilon)T_F(r)+o(T_f(r)).\label{eq2.11}
\end{eqnarray}
Take $\epsilon=\frac{1}{2}$ in \eqref{eq2.11}, and from \eqref{eq2.5} and \eqref{eq3-22}, we obtain
\begin{eqnarray*}
&&\|\ \frac{\Delta N^{\ell+1}}{(\ell+1)!d}\left(1+o(1)\right)\cdot\sum_{j=1}^qm_f(r,D_j)\\
&\le&(m-\ell+1)\left(M+\frac{1}{2}\right)N T_f(r)+o(T_f(r))\\
&=&(m-\ell+1)\left(\frac{\Delta N^{\ell}}{\ell!}+o(N^{\ell})+\frac{1}{2}\right)N T_f(r)+o(T_f(r)),
\end{eqnarray*}
i.e.,
\begin{equation}\label{eq2.13}
\|\ \frac{1}{d}\sum_{j=1}^qm_f(r,D_j)\leq (m-\ell+1)(\ell+1+o(1))T_f(r).
\end{equation}
For $\varepsilon>0$ given in the Main Theorem, take $N$ large enough such that
$
o(1)<\varepsilon
$.
Then
\begin{eqnarray}
\|\frac{1}{d}\sum_{j=1}^qm_f(r,D_j)\le (m-\ell+1)(\ell+1+\varepsilon)T_f(r).\label{eq2.14}
\end{eqnarray}

Since $\ell\leq n\leq m$, we have
\begin{equation*}
(m-\ell+1)(\ell+1)\leq\left\{\begin{array}{cc}
	(\frac{m}{2}+1)^2, &\ m\leq 2n, \\
	(m-n+1)(n+1), &\ m>2n,
	\end{array} \right.
\end{equation*}
which completes the proof of the Main Theorem.

\begin{defi}\label{defi3.2}
	Let $f$ be a nonconstant holomorphic curve in $\mathbb{P}^n({\mathbb{C}})$ with a reduced representation $(f_0,\ldots,f_n)$ and let $\mathcal{D}=\{D_1,\ldots,D_q\}$ be a family of moving hypersurfaces defined by a set of homogeneous polynomials $\mathcal{Q}=\{Q_1,\ldots,Q_q\}$. For an integer $0<\ell\le n$, $f$ is said to be \textbf{$\ell$-nondegenerate} over $\mathcal{K}_{\mathcal{Q}}$ if $\dim V=\ell$. Here $V$ is the variety constructed by all homogeneous polynomials $P\in\mathcal{K}_{\mathcal{Q}}[x_0,\ldots,x_n]$ such that $P(f_0,\ldots,f_n)\equiv 0$.
\end{defi}

By (\ref{eq2.14}), we have the following Cartan-Nochka type second main theorem.

\begin{theorem}\label{thm3.1}
	Let $f$ be a holomorphic curve in $\mathbb{P}^n({\mathbb{C}})$ ($n>1$). Let ${\mathcal D}=\{D_1,\ldots, D_q\}$ be a family of slowly (with respect to $f$) moving hypersurfaces in $m$-subgeneral position, and let ${\mathcal Q}=\{Q_1,\ldots, Q_q\}$ be the set of defining homogeneous polynomials of ${\mathcal D}$ with $\deg Q_j=d_j$ and $Q_j(\fa)\not\equiv 0$ for $j=1,\ldots,q$. Assume that $f$ is $\ell$-nondegenerate over $\mathcal{K}_{\mathcal{Q}}$. Then, for any $\varepsilon>0$,
\begin{eqnarray*}
\|\sum_{j=1}^q \frac{1}{d_j}m_f(r, D_j) \leq (m-\ell+1)(\ell+1+\varepsilon)T_f(r).
\end{eqnarray*}
\end{theorem}

We remark that Theorem \ref{thm3.1} is a generalization of Theorem 1.1 in \cite{15} for moving targets.

%\end{CJK*}

\begin{thebibliography}{99}


\bibitem{4} Cartan H. Sur les zeros des combinaisions linearires
de $p$ fonctions holomorpes donnees.  Mathematica(Cluj),
 7 (1933), 80--103

%\bibitem{11} Chen, ZhiHua; Ru, Min; Yan, QiMing. The degenerated second main theorem and Schmidt's subspace theorem. Sci. China Math. 55 (2012), no. 7, 1367--1380. MR2943780

\bibitem{11} Corvaja, Pietro; Zannier, Umberto. On a general Thue's equation. Amer. J. Math. 126 (2004), no. 5, 1033--1055. MR2089081

\bibitem{9} Dethloff, Gerd; Tran, Van Tan. A second main theorem for moving hypersurface targets. Houston J. Math. 37 (2011), no. 1, 79--111. MR2786547

\bibitem{9-1} Dethloff, Gerd; Tran, Van Tan. Holomorphic curves into algebraic varieties intersecting
moving hypersurfaces targets. preprint.

\bibitem{7} Eremenko, A. E.; Sodin, M. L. Distribution of values of meromorphic functions and meromorphic curves from the standpoint of potential theory. (Russian)  St. Petersburg Math. J. 3 (1992), no. 1, 109--136 MR1120844

%\bibitem{12} Levin, Aaron. Generalizations of Siegel's and Picard's theorems. Ann. of Math. (2) 170 (2009), no. 2, 609--655. MR2552103
\bibitem{10-2} Matsumura, Hideyuki. Commutative algebra. Second edtion. Mathematics Lecture Note Series, 56. Benjamin/Cummings Publishing Co., Inc., Reading, Mass., 1980. xv+313pp. ISBN: 0-8053-7026-9 MR0575344

\bibitem{1} Nevanlinna, Rolf. Le th\'{e}or\`{e}me de Picard-Borel et la th\'{e}orie des fonctions m\'{e}romorphes.
Gauthier-Villars, Paris, 1929.

\bibitem{5} Nochka, E. I. On the theory of meromorphic curves. (Russian) Dokl. Akad. Nauk SSSR 269 (1983), no. 3, 547--552. MR0701289

\bibitem{2} Osgood, Charles F. Sometimes effective Thue-Siegel-Roth-Schmidt-Nevanlinna bounds, or better. J. Number Theory 21 (1985), no. 3, 347--389. MR0814011

\bibitem{13} Ru, Min. Nevanlinna theory and its relation to Diophantine approximation. World Scientific Publishing Co., Inc., River Edge, NJ, 2001. xiv+323 pp. ISBN: 981-02-4402-9. MR1850002

\bibitem{12} Ru, Min. A defect relation for holomorphic curves intersecting hypersurfaces. Amer. J. Math. 126 (2004), no. 1, 215-226. MR2033568

\bibitem{6} Ru, Min; Stoll, Wilhelm. The Cartan conjecture for moving targets. Several complex variables and complex geometry, Part 2 (Santa Cruz, CA, 1989), 477--508, Proc. Sympos. Pure Math., 52, Part 2, Amer. Math. Soc., Providence, RI, 1991. MR1128565

\bibitem{15} Si, Duc Quang. Degeneracy second main theorems for meromorphic mappings into projective varieties with hypersurfaces. to appear in Trans. Amer. Math. Soc.,  arXiv:1610.03951 [math.CV].

\bibitem{8} Si, Duc Quang. Second main theorem for meromorphic mappings with moving hypersurfaces in subgeneral position. arXiv:1610.08456 [math.CV].

%\bibitem{hart} Hartshorne, Robin. Algebraic geometry. Springer-Verlag, 1975. MR0463157

%\bibitem{10-1} Sombra, Martin. Bounds for the Hilbert function of polynomial ideals and for the degrees in the Nullstellensatz. Algorithms for algebra (Eindhoven, 1996). J. Pure Appl. Algebra 117/118 (1997), 565--599. MR1457856

\bibitem{3} Steinmetz, Norbert. Eine Verallgemeinerung des zweiten Nevanlinnaschen Hauptsatzes. (German) [A generalization of Nevanlinna's second main theorem]. J. Reine Angew. Math. 368 (1986), 134--141.

\bibitem{10-4} Valiron, Georges. Sur la d\'{e}riv\'{e}e des fonctions alg\'{e}bro\"{\i}des. (French) Bull. Soc. Math. France 59 (1931), 17--39. MR1504970

\bibitem{10-3} van der Waerden, B. L. Algebra. Vol. II. Based in part on lectures by E. Artin and E. Noether. Translated from the fifth German edition by John R. Schulenberger. Springer-Verlag, New York, 1991. {\rm xii}+284 pp. ISBN: 0-387-97425-3  MR1080173

\bibitem{We46} Weil, Andrew. Foundations of algebraic geometry. Vol. 29. American Mathematical Soc., 1946.

\bibitem{10} Zariski, Oscar. Generalized weight properties of the resultant of $n+1$ polynomials in $n$ indeterminates. Trans. Amer. Math. Soc. 41 (1937), no. 2, 249--265. MR1501900

\bibitem{ZS60} Zariski, Oscar; Samuel, Pierre. Commutative algebra. Vol. II. The University Series in Higher Mathematics. D. Van Nostrand Co., Inc., Princeton, N. J.-Toronto-London-New York 1960.



\end{thebibliography}
\end{document}